\documentclass[10pt]{amsart}

\usepackage[all]{xy}
\usepackage{amsmath, amssymb, hyperref}
\usepackage{tikz}

\newtheorem{thm}{Theorem}[section]

\newtheorem{conj}[thm]{Conjecture}

\theoremstyle{remark}
\newtheorem{remark}[thm]{Remark}

\theoremstyle{definition}

\makeatletter
\renewcommand*\env@matrix[1][*\c@MaxMatrixCols c]{%
  \hskip -\arraycolsep
  \let\@ifnextchar\new@ifnextchar
  \array{#1}}

\newcommand{\ord}{\operatorname{ord}}

\def\qq{\mathbb{Q}}

\def\rr{\mathbb{R}}
\def\zz{\mathbb{Z}}
\def\cc{\mathbb{C}}

\def\mm{\mathcal{M}}

\def\tt{\mathcal{T}}

\def\oo{\mathcal{O}}

\def\ii{\mathcal{I}}
\def\HH{\mathbb{H}}

\def\Im{\mathrm{Im}\,}
\def\Ar{\mathrm{Ar}}

\def\d{\mathrm{d}}

\def\Xbar{\bar{X}}
\def\Ybar{\bar{Y}}

\def\can{\mathrm{can}}
\def\vareps{\varepsilon}
\def\eps{\epsilon}

\def\tr{\mathrm{tr}}

\numberwithin{equation}{section}

\begin{document}

\title[Point-like limit of the hyperelliptic Zhang-Kawazumi invariant]{Point-like limit of the hyperelliptic Zhang-Kawazumi invariant}

\author{Robin de Jong}

\begin{abstract} The behavior near the boundary in the Deligne-Mumford compactification of many functions on $\mm_{h,n}$ can be conveniently expressed using the notion of ``point-like limit'' that we adopt from the string theory literature. In this note we study a function on $\mm_h$ that has been introduced by N.~Kawazumi and S.~Zhang, independently. We show that the point-like limit of the Zhang-Kawazumi invariant in a family of hyperelliptic Riemann surfaces in the direction of any hyperelliptic stable curve exists, and is given by evaluating a combinatorial analogue of the Zhang-Kawazumi invariant, also introduced by Zhang, on the dual graph of that stable curve.  
\end{abstract}

\maketitle

\thispagestyle{empty}

\section{Introduction}

The invariant from the title has been introduced around 2008   by N. Kawazumi \cite{kawhandbook} \cite{kaw} and S. Zhang \cite{zhgs}, independently, and can be defined as
\begin{equation} \label{definephi} \varphi(\Sigma) = \int_{\Sigma \times \Sigma } g_\Ar(x,y) \,  \nu^2(x,y) \, . 
\end{equation}
Here $\Sigma$ is a compact connected Riemann surface of positive genus, and $g_\Ar$ is the Arakelov-Green's function \cite{ar} on $\Sigma \times \Sigma$. Let $\Delta$ be the diagonal divisor on $\Sigma \times \Sigma$. Then $\nu \in A^{1,1}(\Sigma \times \Sigma)$ is the curvature form of the hermitian line bundle $\oo_{\Sigma \times \Sigma}(\Delta)$ whose metric is given by the prescription
\[ \log\|1\|(x,y) = g_\Ar(x,y)  \] 
for all $(x,y) \in \Sigma \times \Sigma$ away from $\Delta$.

In \cite{kawhandbook} \cite{kaw} the invariant $\varphi$ arises in the context of a study of the first extended Johnson homomorphism \cite{mo} on the mapping class group of a pointed compact connected oriented topological surface. The results from (the unfortunately unpublished) \cite{kaw} were revisited in \cite{djtorus}. The motivation in \cite{zhgs} to study $\varphi$ comes from number theory, where $\varphi$ appears as a local archimedean contribution in a formula that relates the height of the canonical Gross-Schoen cycle on a smooth projective and geometrically connected curve with semistable reduction over a number field with the self-intersection of its admissible relative dualizing sheaf. The connection between these two seemingly different approaches was established in \cite{djnormal}. 

The Zhang-Kawazumi invariant vanishes identically in genus $h=1$. In genera $h \geq 2$, the invariant is strictly positive, cf. \cite[Corollary 1.2]{kaw} or \cite[Proposition~2.5.1]{zhgs}. Several expressions for the Levi form of $\varphi$ are derived in \cite{djtorus} \cite{kaw}. 

The invariant in genus $h=2$ has recently attracted attention from superstring theory \cite{dhgr} \cite{dhgrpr} where its integral against the volume form $\d  \mu_2$ of the Siegel metric over $\mm_2$ appears in the low energy expansion of the two-loop four-graviton amplitude. The detailed study of the invariant in \cite{dhgr} \cite{dhgrpr} has yielded the important result that $\varphi$ is an eigenfunction of the Laplace-Beltrami operator with respect to $\d \mu_2$, with eigenvalue $5$ (cf. \cite[Section~4.6]{dhgrpr}). In \cite{pio}, based on this result a completely explicit expression is given for the genus-two Zhang-Kawazumi invariant as a theta lift, involving a Siegel-Narain theta series and a weight $-5/2$ vector-valued modular form appearing in the theta series decomposition of the weak Jacobi form $\vartheta_1^2(\tau,z)/\eta(\tau)^6$.

One of the main starting points in deriving these results is a study of the low order asymptotics of the Zhang-Kawazumi invariant near the boundary of $\mm_2$ in the Deligne-Mumford compactification $\overline{\mm}_2$. As it turns out, there is a natural combinatorial description of these asymptotics in terms of the tautological stratification of $\overline{\mm}_2$ by topological type of the dual graphs of stable curves. The following theorem summarizes the results from \cite{dhgr} \cite{dhgrpr} about the asymptotics of $\varphi$ in genus $h=2$. We refer to Section \ref{reviewgenustwo} for an explanation of the notation and terminology.  
\begin{thm} \label{genustwo} Let $\Xbar_0$ be a complex stable curve of arithmetic genus two. Let $(G,q)$ be its dual graph and let $\{e_1,\ldots,e_r\}$ denote the edge set of $G$. Let $0 \in U \subset \cc^{3}$ be the universal deformation space of $\bar{X}_0$ as an analytic stable curve, let $\pi \colon \Xbar \to U$ be the associated Kuranishi family where $\bar{X}_0$ is the fiber of $\pi$ at $0$ and suppose that $u_1\cdots u_r$ is an equation for the (reduced normal crossings) divisor $D$ in $U$ given by the points $u \in U$ such that $\Xbar_u$ is not smooth. We assume that the choice of coordinates in $U \subset \cc^3$ is compatible with the given ordering of the edges of $G$. Let $\varphi^\tr(G,q) \in \qq(x_1,\ldots,x_r)$ be the rational function given in Table \ref{phiinv}. 
\begin{table} \label{phiinv}
\caption{Point-like limit of the ZK invariant in genus two}
\begin{center}
\begin{tabular}{|l|l|}
\hline
$(G,q)$ & $\varphi^{\tr}(G,q)$  \\  [2pt]
\hline
$I(e_1,e_2,e_3)$  & $\frac{1}{12}(x_1+x_2+x_3) - \frac{5}{12}\frac{x_1x_2x_3}{x_1x_2+x_2x_3+x_3x_1}$ \\ [6pt]
$II(e_1)$  & $x_1$ \\ [4pt]
$III(e_1)$  & $\frac{1}{12}x_1$ \\ [4pt]
$IV(e_1,e_2)$  & $x_1 + \frac{1}{12}x_2$ \\ [4pt]
$V(e_1,e_2)$  & $\frac{1}{12}(x_1  + x_2)$ \\ [4pt]
$VI(e_1,e_2,e_3)$  & $x_1 + \frac{1}{12}(x_2 + x_3)$   \\ [4pt]
\hline
\end{tabular}
\end{center}
\end{table}
Then the asymptotics
\begin{equation} \label{asymptgenustwo} \varphi(\Xbar_u) = \varphi^\tr(G,q,-\log|u_1|,\ldots,-\log|u_r|) + O(1) 
\end{equation}
holds as $u \to 0$ over $U \setminus D$.
\end{thm}
In particular, note that for any holomorphic arc $f \colon \Delta \to U$ with $f(0)=0$ and with image not contained in $D$ we obtain from (\ref{asymptgenustwo}) that
\begin{equation} \label{ptlikelimit}
\alpha' \varphi(\Xbar_{f(t)}) = \varphi^\tr(G,q,m_1,\ldots,m_r) + O(\alpha') 
\end{equation}
as $t \to 0$ in $\Delta^*$, where $\alpha'=-(\log|t|)^{-1}$ and $m_i = \ord_0 f^*u_i$ for $i=1,\ldots,r$. In the language of string theory, equation (\ref{ptlikelimit}) says that the ``point-like limit'' of $\varphi$ in the direction of $\Xbar_0$ exists, and is equal to $\varphi^\tr(G,q)$. The notation ${}^\tr$ stands for ``tropical'' and is borrowed from \cite{to} where, for a variety of string integrands on $\mm_{h,n}$, the point-like limit is studied from the point of view of tropical modular geometry. The point-like limit of the N\'eron height pairing on a pair of degree zero divisors in arbitrary genus is computed, and expressed in terms of the combinatorics of the dual graph, in the forthcoming paper \cite{abbf}, see also \cite{bloch}.

Theorem \ref{genustwo} is proved by a case-by-case analysis. One obvious question is whether there is a general mechanism that would produce the various $\varphi^\tr(G,q)$ from Table \ref{phiinv}. In particular, one would like to predict the point-like limits of $\varphi$ in genera larger than two (assuming they exist).

In \cite{zhgs} S. Zhang introduced and studied an invariant $\varphi(G,q)$ for polarized graphs $(G,q)$ that takes values in $\qq(x_1,\ldots,x_r)$ if $G$ has $r$ ordered edges. The invariant $\varphi(G,q)$ is homogeneous of weight one in the variables $x_1,\ldots,x_r$, and serves alongside the Zhang-Kawazumi invariant $\varphi(\Sigma)$ in a formula that relates the height $(\Delta_\xi,\Delta_\xi) \in \rr$ of the so-called canonical Gross-Schoen cycle $\Delta_\xi$ of a smooth, projective and geometrically connected curve $X$ with semistable reduction over a number field with the self-intersection $(\omega_a,\omega_a) \in \rr$ of its admissible relative dualizing sheaf, cf. \cite[Corollary 1.3.2]{zhgs}. We have calculated the graph invariant $\varphi(G,q)$ for polarized graphs of genus two in \cite[Theorem 2.1]{djgenus2}. A quick comparison between Table 1 in \cite{djgenus2} and Table \ref{phiinv} above yields that, perhaps rather strikingly, for each of the topological types I--VI in genus two one has that $\varphi(G,q)=\varphi^\tr(G,q)$. 

We expect that this is not just a coincidence. That is, apart from being a suitable non-archimedean analogue of the Zhang-Kawazumi invariant, the invariant $\varphi(G,q)$ should also be the point-like limit of the Zhang-Kawazumi invariant in the direction of each stable curve with dual graph $(G,q)$, in the following sense.
\begin{conj} \label{conjasympt}
Let $\Xbar_0$ be a complex stable curve of arithmetic genus $h \geq 2$. Let $(G,q)$ be the polarized dual graph of $\Xbar_0$, and write $E(G)=\{e_1,\ldots,e_r\}$ for the edge set of $G$. Let $0 \in U \subset \cc^{3h-3}$ be the universal deformation space of $\bar{X}_0$ as an analytic stable curve,  and let $\pi \colon \Xbar \to U$ be the associated Kuranishi family such that $\bar{X}_0$ is the fiber of $\pi$ at $0$. Suppose that $u_1\cdots u_r$ is an equation for the divisor $D$ in $U$ given by the points $u \in U$ such that $\Xbar_u$ is not smooth. We assume that the choice of coordinates in $U \subset \cc^{3h-3}$ is compatible with the given ordering of the edges of $G$. Let $f \colon \Delta \to U$ be a holomorphic arc with $f(0)=0$, and with image not contained in $D$. Let $\varphi(G,q) \in \qq(x_1,\ldots,x_r)$ be Zhang's graph invariant of $(G,q)$. Then one has the asymptotics
\begin{equation} \label{ptlikelimitbis}
\alpha' \varphi(\Xbar_{f(t)}) = \varphi(G,q,m_1,\ldots,m_r) + O(\alpha') 
\end{equation}
as $t \to 0$ in $\Delta^*$, where $\alpha'=-(\log|t|)^{-1}$ and $m_i = \ord_0 f^*u_i$ for $i=1,\ldots,r$.
\end{conj}

To obtain a feeling for the strength of this conjecture, we note that since the Zhang-Kawazumi invariant in genus $h \geq 2$ is always positive, and any polarized stable graph of genus $h \geq 2$ is the dual graph of some stable curve of genus $h$, we immediately obtain from (\ref{ptlikelimitbis}) that for any polarized stable graph $(G,q)$ of genus $h \geq 2$, and any tuple $(m_1,\ldots,m_r)$ of positive integers with $r=|E(G)|$, the rational number $\varphi(G,q,m_1,\ldots,m_r)$ should be non-negative. This statement is known to be true as follows from \cite[Theorem~2.11]{ci}, however the proof of the inequality in \cite[Theorem~2.11]{ci} seems to be rather involved. 

On the positive side, we know that in all cases where $r=1$ Conjecture \ref{conjasympt} is true. This follows from comparing \cite[Corollary 1.6]{djsecond} and \cite[Propositions 4.4.1 and 4.4.3]{zhgs}. The asymptotic result of \cite[Corollary 1.6]{djsecond} has been refined in \cite[Theorem A]{djasympt}. We note that the cases where $r=1$ correspond to the generic points of the boundary divisor of $\mm_h$ in $\overline{\mm}_h$.

The aim of the present paper is to prove the following partial result, where we limit ourselves to hyperelliptic curves (of any genus). When $\psi_1,\psi_2$ are continuous functions on the punctured disc $\Delta^*$, we write $\psi_1\sim \psi_2$ if the difference $\psi_1-\psi_2$ extends as a continuous function over $\Delta$. 

\begin{thm} \label{main} Let $\Xbar_0$ be a complex hyperelliptic stable curve of arithmetic genus $h \geq 2$. Let $(G,q)$ be the polarized dual graph of $\Xbar_0$, and write $E(G)=\{e_1,\ldots,e_r\}$. Let $0 \in U \subset \cc^{2h-1}$ be the universal deformation space of $\bar{X}_0$ as an analytic hyperelliptic stable curve, and let $\pi \colon \Xbar \to U$ be the associated Kuranishi family where $\bar{X}_0$ is the fiber of $\pi$ at $0$. Suppose that $u_1\cdots u_r$ is an equation for the (reduced normal crossings) divisor $D$ in $U$ given by the points $u \in U$ such that $\Xbar_u$ is not smooth. Let $f \colon \Delta \to U$ be a holomorphic arc with $f(0)=0$, and with image not contained in $D$. Let $\varphi(G,q) 
\in \qq(x_1,\ldots,x_r)$ be Zhang's graph invariant of $(G,q)$. Then one has the asymptotics
\begin{equation} \label{ptlikelimitthrice}
\varphi(\Xbar_{f(t)}) \sim -\varphi(G,q,m_1,\ldots,m_r)\log |t|
\end{equation}
as $t \to 0$ in $\Delta^*$, where $m_i = \ord_0 f^*u_i$ for $i=1,\ldots,r$.
\end{thm}
Note that the estimate in (\ref{ptlikelimitthrice}) is stronger than the one claimed in (\ref{ptlikelimitbis}). We also note that Theorem \ref{main} gives a uniform explanation for the six occurring point-like limits in equation (\ref{ptlikelimit}). However, the result seems not strong enough to produce the $O(1)$-symbol in (\ref{asymptgenustwo}). 

The contents of this paper are as follows. For the sake of concreteness, and also to explain the notation used in Table \ref{phiinv}, we will first discuss the proof of Theorem \ref{genustwo} at some length, in Section \ref{reviewgenustwo} below. Then in Section \ref{prelims} we review the definition of Zhang's graph invariant $\varphi(G,q)$, and recall some of its properties. In the final Section \ref{proofmain} we present our proof of Theorem \ref{main}. 

\section{Proof of Theorem \ref{genustwo}} \label{reviewgenustwo}

As noted before, Theorem \ref{genustwo} follows by gathering together various results from \cite{dhgr} \cite{dhgrpr} about the asymptotics of the Zhang-Kawazumi invariant in genus two. In this section, we adapt these results to our situation and provide some details about the $O(1)$-symbol in (\ref{asymptgenustwo}). 

A polarized graph is a finite connected graph $G$ together with a map $q \colon V(G) \to \zz_{\geq 0}$, called the polarization. Here $V(G)$ denotes the vertex set of $G$. Let $(G,q)$ be a polarized graph, with Betti number $b_1(G)$. We call the non-negative integer $h(G,q) = b_1(G) + \sum_{x \in V(G)} q(x)$ the genus of $(G,q)$. The polarized graph $(G,q)$ is called stable if for each vertex $x$ with $q(x)=0$ the number of half-edges emanating from $x$ is at least $3$. For any given non-negative integer $h$, there are, up to isomorphism, only finitely many stable polarized graphs $(G,q)$ of genus $h(G,q)=h$. If $h=2$, there are precisely seven types. Apart from the trivial polarized graph (one vertex with $q=2$, no edges), we list these types below as Cases I-VI. 

Let $\Xbar_0$ be a complex stable curve. The dual graph $G$ of $\Xbar_0$ is a finite connected graph whose vertex set $V(G)$ consists of the irreducible components of $\Xbar_0$, and whose edge set $E(G)$ consists of the singular points of $\Xbar_0$. The vertex assignment map is given by sending a singular point $e \in E(G)$ to the set of irreducible components of $\Xbar_0$ that $e$ lies on. The dual graph of $\Xbar_0$ has a canonical polarization $q$ given by assigning to each $x \in V(G)$ the genus of the normalization of $x$. The associated polarized graph $(G,q)$ is stable. If $\Xbar_0$ has arithmetic genus $h$, then the genus of $(G,q)$ is also equal to~$h$. 

Now let $\Xbar_0$ be a complex stable curve of arithmetic genus two, with (polarized, stable) dual graph $(G,q)$ and with $E(G)=\{e_1,\ldots,e_r\}$. Let $0 \in U \subset \cc^{3}$ be the universal deformation space of $\bar{X}_0$ as an analytic stable curve, let $\pi \colon \Xbar \to U$ be the Kuranishi family such that $\bar{X}_0$ is the fiber of $\pi$ at $0$, and suppose that $u_1\cdots u_r=0$ defines the divisor $D$ of $U$ given by the points $u \in U$ such that $\Xbar_u$ is not smooth. 

In \cite{dhgr} \cite{dhgrpr} one considers period matrices 
\[ \Omega = \left( \begin{array}{cc} \tau_1 & \tau \\ \tau & \tau_2 \end{array} \right) \, , \quad \Im \Omega>0 \]
of genus two curves near the boundary of Siegel upper half space $\HH_2$ of degree two. In particular $\tau_1, \tau_2$ are elements of the upper half plane $\HH_1$. In order to obtain the asymptotics as stated in Theorem \ref{genustwo}, we need to translate the asymptotics from \cite{dhgr} \cite{dhgrpr}, which are phrased in terms of $\Omega$, in terms of the standard coordinates $u_1,u_2,u_3$ on the universal deformation space $U$. In other words, we first need to write $\Omega=\Omega(u)$ explicitly in coordinates $u_1, u_2, u_3$. In those cases where the period matrices associated to $\Xbar_u$ are unbounded, one finds an expansion of $\Omega$ in terms of $u_1, u_2, u_3$ using a suitable several variables version of the Nilpotent Orbit Theorem, cf. \cite{ca}.

We recall that there are six topological types I-VI of (non-trivial) polarized stable genus two graphs to deal with. We discuss them case by case. For future reference we have assigned lengths $m_1, \ldots$ to the edges of each graph. In the following $\vartheta_1(z,\tau)$ denotes Riemann's theta function with modulus $\tau$ in genus one (in particular we have $z \in \cc$, $\tau \in \HH_1$). Finally $\eta(\tau)$ denotes Dedekind's eta function. \\

Case I -- This is the case of the ``sunset'' graph,
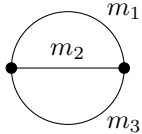
\begin{center}
\begin{tikzpicture}
\draw (0,0) circle [radius=0.75];
\draw (-0.75,0) node[circle, draw, fill=black, inner sep=0pt, minimum width=4pt]{} -- (0.75,0) node[circle, draw, fill=black, 
inner sep=0pt, minimum width=4pt]{};
\node at (0.75,0.75) {$m_1$};
\node [above] at (0,0) {$m_2$};
\node at (0.75,-0.75) {$m_3$};
\end{tikzpicture}
\end{center}
Both vertices have $q=0$. The corresponding stable curve $\Xbar_0$ consists of two projective lines, joined together at three points. If $0 \in U \subset \cc^3$ is the universal deformation space of $\Xbar_0$, then an analysis of the monodromy on the general fiber of $\Xbar \to U$ together with the Nilpotent Orbit theorem \cite{ca} gives 
\[ \Omega(u_1,u_2,u_3) = \left( \begin{array}{cc} 1  & 0  \\ 0  &  0 \end{array} \right)\frac{\log u_1}{2\pi i} + \left( \begin{array}{cc} 0  & 0  \\ 0  & 1  \end{array} \right) \frac{\log u_2}{2\pi i} + \left( \begin{array}{cc} 1  & 1  \\  1 & 1  \end{array} \right)\frac{\log u_3}{2\pi i} \, , \]
up to a bounded holomorphic function. We obtain 
\[ \Im \Omega = \left( \begin{array}{cc} L_1+L_3 & L_3 \\ L_3 & L_2+L_3 \end{array} \right) + O(1) \, , \quad L_i = -\frac{1}{2\pi}\log|u_i| \, , \quad i=1, 2, 3 \, ,  \]
and with this notation then following \cite[Section~3.2]{dhgrpr} we have the asymptotics
\begin{equation} \label{supergravlimit}  \varphi(\Omega) = \frac{\pi}{6} \left[ L_1+L_2+L_3 - \frac{5L_1L_2L_3}{L_1L_2+L_2L_3+L_3L_1} \right] + O(1/L_i^2) 
\end{equation}
for the Zhang-Kawazumi invariant.
It follows that
\[ \begin{split} \varphi(\Xbar_u) = -\frac{1}{12}(&
\log|u_1|+\log|u_2|+\log|u_3| ) \\ & + \frac{5}{12} \frac{\log|u_1|\log|u_2|\log|u_3|   }{\log|u_1|\log|u_2| + \log|u_2|\log|u_3| + \log|u_3|\log|u_1|   }  + O(1)
\end{split} \]
as $t \to 0$ through $U \setminus D$. Comparing with Table \ref{phiinv} we see that in this case (\ref{asymptgenustwo}) is satisfied.

Case II -- This is the ``minimal separating degeneration limit'', obtained by letting $\tau \to 0$ while keeping $\tau_1, \tau_2 \in \HH_1$ bounded away from zero and infinity. The corresponding stable curve $\Xbar_0$ is the join of the elliptic curves with moduli $\tau_1(0)$ and $\tau_2(0)$. The dual graph $G$ looks like
\begin{center}
\begin{tikzpicture}
\draw node[circle, draw, fill=black, 
inner sep=0pt, minimum width=4pt]{} (0,0) -- (1,0) node[circle, draw, fill=black, 
inner sep=0pt, minimum width=4pt]{};
\node [above] at (0.5,0) {$m_1$};
\end{tikzpicture}
\end{center}
where both vertices have $q=1$. 
Following \cite[Theorem A or Section 5]{djasympt} or \cite[Section 6.1]{dhgr} we then have
\begin{equation} \label{minsep} \varphi = -\log|2\pi  \tau \eta^2(\tau_1) \eta^2(\tau_2)| + O(\tau^2 \log|\tau| )  
\end{equation}
as $\tau \to 0$. In this case we can identify $\tau=u_1$ so that we obtain the asymptotics
\[  \varphi = - \log|u_1| + O(1) \]
as $t \to 0$. We thus obtain (\ref{asymptgenustwo}) also in this case.

Case III -- This is the ``minimal non-separating degeneration limit'', obtained by sending $\tau_2 \to i \infty$ while keeping $\tau_1 \in \HH$ and $\tau \in \cc$ bounded away from zero and infinity. The corresponding stable curve $\Xbar_0$ is the elliptic curve with modulus $\tau_1(0)$ with the points determined by $z=0$ and $z=\tau(0)$ identified. The dual graph $G$ looks like
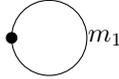
\begin{center}
\begin{tikzpicture}
\draw (0,0) circle [radius=0.5];
\draw (-0.5,0) node[circle, draw, fill=black, 
inner sep=0pt, minimum width=4pt]{};
\node at (0.75,0) {$m_1$};
\end{tikzpicture}
\end{center}
where the unique vertex has $q=1$. Following \cite[Theorem A or Section 5]{djasympt} or \cite[Section 6.1]{dhgr} we have
\begin{equation} \label{minnonsep} \varphi = \frac{\pi}{6}  \left( \Im \tau_2 +\frac{5(\Im \tau)^2}{\Im \tau_1} \right) - \log \left| \frac{\vartheta_1(\tau,\tau_1)}{\eta(\tau_1)} \right| + O(1/\Im \tau_2) 
\end{equation}
as $\tau_2 \to i \infty$. In terms of the local coordinates $u_1,u_2,u_3$ of $U$ we have, by the Nilpotent Orbit theorem,
\[  \tau_2 = \frac{ \log u_1}{2\pi i }  \]
up to a bounded holomorphic function of $u$. We obtain the asymptotics
\[  \varphi(\Xbar_u) = -\frac{1}{12} \log |u_1| + O(1) \]
as $t \to 0$, thus verifying (\ref{asymptgenustwo}) again.\\

The three remaining cases can be obtained by successively further degenerating the elliptic curves in Cases II and III above.
One then obtains (cf. \cite[Section~6.2]{dhgr}) the required asymptotics of $\varphi$ from the previous two cases by a careful study of the stated subleading terms. \\

Case IV -- Here we take a stable curve as in Case II but further degenerate the modulus $\tau_2$ by sending $\tau_2 \to i \infty$. The resulting stable curve is the join of an elliptic curve with modulus $\tau_1(0)$ and a rational curve with points $0$ and $\infty$ identified. The dual graph $G$ looks like
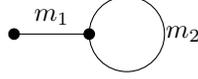
\begin{center}
\begin{tikzpicture}
\draw (-1,0) node[circle, draw, fill=black, inner sep=0pt, minimum width=4pt]{} -- (0,0) node[circle, draw, fill=black, 
inner sep=0pt, minimum width=4pt]{};
\draw (0.5,0) circle [radius=0.5];
\node at (1.25,0) {$m_2$};
\node [above] at (-0.5,0) {$m_1$};
\end{tikzpicture}
\end{center}
The shape of the degeneration is 
\[ \tau \to 0 \, , \, \tau_2 \to i \infty  \, , \quad
\varphi(\Omega) = -\log|\tau| + \frac{\pi}{6} \Im \tau_2 + O(1) \, , \]
and this leads to the asymptotics
\[ \varphi(\Xbar_u)=-\log|u_1| -\frac{1}{12}\log |u_2| + O(1)
\, .  \]

Case V -- We further degenerate the curve in Case III by sending $\tau_1 \to i \infty$. The resulting stable curve is a rational curve with two double points. Its dual graph looks like
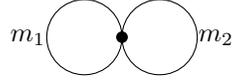
\begin{center}
\begin{tikzpicture}
\draw (-0.5,0) circle [radius=0.5];
\draw (0.5,0) circle [radius=0.5];
\draw (0,0) node[circle, draw, fill=black, inner sep=0pt, minimum width=4pt]{};
\node at (-1.25,0) {$m_1$};
\node at (1.25,0) {$m_2$};
\end{tikzpicture}
\end{center}
where the vertex has $q=0$ and we have
\[ \tau_1 \to i \infty  \, , \, \tau_2 \to i \infty  \, , \quad
\varphi(\Omega) =  \frac{\pi}{6} \Im \tau_1 + \frac{\pi}{6} \Im \tau_2 + O(1) \, , \]
so that 
\[\varphi(\Xbar_u)=
-\frac{1}{12}\log|u_1| -\frac{1}{12}\log |u_2| + O(1) \, . 
 \]

Case VI -- In this final case we further degenerate the elliptic curve with modulus $\tau_1$ from Case IV by sending $\tau_1 \to i \infty$. The resulting stable curve is the join of two rational curves each carrying a single double point. The dual graph of $\Xbar_0$ in this case looks like
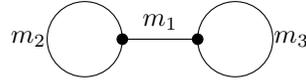
\begin{center}
\begin{tikzpicture}
\draw (-1,0) circle [radius=0.5];
\draw (1,0) circle [radius=0.5];
\draw (-0.5,0) node[circle, draw, fill=black, inner sep=0pt, minimum width=4pt]{} -- (0.5,0) node[circle, draw, fill=black, inner sep=0pt, minimum width=4pt]{};
\node at (-1.75,0) {$m_2$};
\node at (1.75,0) {$m_3$};
\node [above] at (0,0) {$m_1$};
\end{tikzpicture}
\end{center}
where both vertices have $q=0$. We have
\[ \begin{split} \tau \to 0 \, , \,  \tau_1 \to i \infty  \, , \, & \tau_2 \to i \infty \, , \\
& \varphi(\Omega) = -\log|\tau| + \frac{\pi}{6} \Im \tau_1 + \frac{\pi}{6}\Im \tau_2 +O(1) \, , \end{split} \]
so that
\[ \varphi(\Xbar_u)=
-\log|u_1| -\frac{1}{12}\log |u_2| - \frac{1}{12}\log|u_3| +O(1) \, . 
 \]

Theorem \ref{genustwo} is hereby proved. In closing this section we mention that more refined versions of the asymptotics in (\ref{supergravlimit}), (\ref{minsep}) and (\ref{minnonsep}) have been determined by B. Pioline \cite{pio}. 

\section{Polarized metric graphs and Zhang's graph invariant} \label{prelims}

The aim of this section is to recall from \cite{br} \cite[Appendix]{zh} some notions related to metric graphs, and to introduce Zhang's graph invariant $\varphi(\Gamma,q)$ of a polarized metric graph $(\Gamma,q)$.

A metric graph $\Gamma$ is a compact connected metric space such that for each $x \in \Gamma$ there exist a natural number $n$ and $\eps \in \rr_{>0}$ such that $x$ has a neighborhood isometric to the set 
\[ S(n,\eps)= \{ z \in \cc \, : \, z = t e^{2 \pi i k/n} \,\, \textrm{for some} \,\, 0 \leq t < \eps \,\, \textrm{and some} \,\, k \in \zz \} \, , \]
endowed with its natural metric. If $\Gamma$ is a metric graph, then for each $x \in \Gamma$ the integer $n$ is uniquely determined, and is called the valence of $x$. We denote by $V_0 \subset \Gamma$ the set of points $x \in \Gamma$ with valency $\neq 2$. This is a finite set. Any finite non-empty set $V \subset \Gamma$ containing $V_0$ is called a vertex set of $\Gamma$.

If $V$ is a vertex set of $\Gamma$ then $\Gamma \setminus V$ is a finite union of open intervals. The closure of a connected component of $\Gamma \setminus V$ is called an edge associated to $V$. Let $E$ be the set of edges associated to $V$. For $e \in E$ we call $e\setminus e^o$ the set of endpoints of $e$. This is a finite set consisting of either one (when $e$ is a loop) or two (when $e$ is a closed interval) elements. 

Let $\Gamma$ be a metric graph, and assume a vertex set $V$ is given, with associated edge set $E$. For an edge $e \in E$ we let $\d y$ denote the usual Lebesgue measure on $e$, and $m(e)$ the volume of $e$. 

Elements of $\rr^V$ are called divisors on $\Gamma$ (with respect to the given vertex set $V$). For a divisor $D=\sum_{x \in V} a(x) x$ we define the degree of $D$ as $\deg D= \sum_{x \in V} a(x)$. Let $C(\Gamma)$ be the set of $\rr$-valued continuous functions on $\Gamma$ that are smooth outside $V$ and have well-defined derivatives at each $v \in V$ along each $e \in E$ emanating from $v$. Denote by $C(\Gamma)^*$ the set of linear functionals on $C(\Gamma)$. The elements of $C(\Gamma)^*$ are called currents on $\Gamma$. As an important example, each $x \in \Gamma$ gives rise to a Dirac current $\delta_x \in C(\Gamma)^*$ given by sending $f \mapsto f(x)$. Integration of currents over $\Gamma$ gives a natural linear map $C(\Gamma)^* \to \rr$. Following \cite[Section 1.2]{br} or \cite[Appendix]{zh} there exists a natural Laplace operator $\Delta \colon C(\Gamma) \to C(\Gamma)^*$.  
 
Let $K_\can$ be the divisor on $\Gamma$ given by $K_\can(x)= v(x)-2$ for each $x \in V$. Note that $\deg K_\can = 2b_1(\Gamma)-2$, where $b_1(\Gamma)$ is the Betti number of $\Gamma$. For an edge $e \in E$, we denote by $r(e)$ the effective resistance between the endpoints of $e$ in $\Gamma \setminus e^o$, where $\Gamma$ is viewed as an electric circuit with edge resistances given by the $m(e)$. We set $r(e)$ to be $\infty$ if $\Gamma \setminus e^0$ is disconnected. We have a canonical current \cite{cr} 
\[ \mu_{\can} = -\frac{1}{2}\delta_{K_\can} + \sum_{e \in E} \frac{\d y}{m(e) + r(e)} \]
in $C(\Gamma)^*$. By \cite[Theorem 2.11]{cr} the current $\mu_\can$ is a probability measure, that is we have $\int_\Gamma \mu_{\can} = 1$. 

A polarization on $\Gamma$ is a map $q \colon V \to \zz_{\geq 0}$. Associated to a polarization $q$ we consider the divisor
\[ K_q = K_\can + 2 \sum_{x \in V} q(x) \cdot x  \]
on $\Gamma$. Putting $h(\Gamma,q)=b_1(\Gamma) + \sum_{x \in V} q(x)$ we see that $\deg K_q = 2h(\Gamma,q)-2$. We call $h(\Gamma,q)$ the genus of the polarized metric graph $(\Gamma,q)$. For polarized metric graphs $(\Gamma,q)$ of positive genus $h$, Zhang's admissible measure is defined to be the current
\[ \mu = \mu(\Gamma,q) = \frac{1}{2h} \left( \delta_{K_q} + 2 \mu_{\can} \right)  \]
in $C(\Gamma)^*$. We note that $\int_\Gamma \mu = 1$. Finally, Zhang's Arakelov-Green's function $g_\mu(x,y)$ on $\Gamma \times \Gamma$ is determined by the conditions 
\[ \Delta_y g_\mu(x,y) = \delta_x(y)-\mu(y) \, , \quad \int_\Gamma g_\mu(x,y) \mu(y) = 0  \]
for all $x \in \Gamma$. We refer to \cite[Section 3 and Appendix]{zh} or \cite[Section 1.5]{br} for a proof that $g_\mu(x,y)$ exists in $C(\Gamma)$ for all $x \in \Gamma$, as well as for a discussion of its main properties.

Let $(\Gamma,q)$ be a polarized metric graph of genus $h$. We denote by $\delta(\Gamma)$ the total volume of $\Gamma$. Zhang's graph invariant $\varphi(\Gamma,q)$ is defined via the formula \cite[Section~1.3]{zhgs}
\begin{equation} \label{graphphi} \varphi(\Gamma,q) = -\frac{1}{4}\delta(\Gamma) + \frac{1}{4} \int_\Gamma g_\mu(x,x)((10h+2)\mu(x) - \delta_{K_q}(x)) \, .   
\end{equation}

For later reference we define the similar epsilon-invariant \cite[Theorem~4.4]{zh} 
\begin{equation} \label{defeps} \varepsilon(\Gamma,q) = \int_\Gamma g_\mu(x,x)((2h-2)\mu(x)+\delta_{K_q}(x)) \, . 
\end{equation}

Let $G$ be a finite connected graph with set of edges $E(G)$ and set of vertices $V(G)$. Then to $G$ we have naturally associated a metric graph by glueing $|E(G)|$ closed intervals $[0,1]$ according to the vertex assignment map. More generally, if a weight $m \in \rr_{>0}^{E(G)}$ of $E(G)$ is given, then one has an associated metric graph $\Gamma=(G,m)$ by glueing the closed intervals $[0,m(e)]$, where $e$ runs through $E(G)$, according to the vertex assignment map. Usually we assume that the edge set $E(G)$ is ordered, so that we may write $\Gamma=(G,m_1,\ldots,m_r)$ where $r=|E(G)|$. Note that $\Gamma$ comes with a natural vertex set $V(G)$. A polarization $q \colon V(G) \to \zz_{\geq 0}$ of $G$ naturally gives rise to a polarized metric graph $(\Gamma,q)=(G,q,m_1,\ldots,m_r)$. 

We see that if $(G,q)$ is a polarized graph, then for each $m \in \rr_{>0}^{E(G)}$ we have an element $\varphi(G,q,m) \in \rr$. We thus obtain a natural map $\varphi(G,q) \colon \rr_{>0}^{E(G)} \to \rr$. By \cite[Proposition 4.6]{fa} we have the following important property of $\varphi(G,q)$. Let $b_1(G)$ be the first Betti number of $G$. Then there exist homogeneous integral polynomials $P$, $Q$ of degree $2b_1(G)+1$ resp. $2b_1(G)$ such that 
\[ \varphi(G,q,m_1,\ldots,m_r)=P(m_1,\ldots,m_r )/Q(m_1,\ldots,m_r) \]
 for all $(m_1,\ldots,m_r) \in \rr_{>0}^r$. In particular, the map $\varphi(G,q) \colon \rr_{>0}^{E(G)} \to \rr$ is continuous, and homogeneous of weight one. Moreover, we may view $\varphi(G,q)$ as an element of the rational function field $\qq(x_1,\ldots,x_r)$. 

In \cite{cicomp} \cite{fab} one finds algorithms to calculate $\varphi(G,q)\in \qq(x_1,\ldots,x_r)$, together with many examples in low genera. In \cite{djgenus2} a list is given of the $\varphi(G,q)$ for all stable polarized graphs of genus two. In \cite{ci3} a complete study is made of $\varphi(G,q)$ for all stable polarized graphs of genus three. In \cite{yacorn} \cite{ya} the invariant $\varphi(G,q)$ is studied for so-called hyperelliptic polarized graphs. In our proof of Theorem \ref{main} below we will make essential use of a result from \cite{ya} that relates $\vareps(G,q)$ and $\varphi(G,q)$ for hyperelliptic polarized graphs. In \cite{ci} an effective lower bound is proven for the invariant $\varphi(\Gamma,q)$ for general polarized metric graphs, establishing Conjecture 1.4.2 from \cite{zhgs}.

\section{Proof of Theorem \ref{main}} \label{proofmain}

As in Theorem \ref{main}, let $\Xbar_0$ be a complex hyperelliptic stable curve of arithmetic genus $h \geq 2$, with canonically polarized dual graph $(G,q)$. Let $0 \in U \subset \cc^{2h-1}$ be the universal deformation space of $\bar{X}_0$ as an analytic hyperelliptic stable curve, and let $\pi \colon \Xbar \to U$ be the associated Kuranishi family where $\bar{X}_0$ is the fiber of $\pi$ at $0$. Assume that the locus $D$ of points $u \in U$ such that $\Xbar_u$ is not smooth is given by the equation $u_1\cdots u_r$, in particular $|E(G)|=r$. We consider a holomorphic arc $f \colon \Delta \to U$, with $f(0)=0$, and with image not contained in $D$. We denote by $\varphi(G,q) \in \qq(x_1,\ldots,x_r)$ Zhang's graph invariant of $(G,q)$. 

Let $m_i = \ord_0 f^*u_i$ for $i=1,\ldots,r$. Then as in Section \ref{prelims} we let $(G,q,m_1,\ldots,m_r)$ denote the polarized metric graph associated to $(G,q)$ and the weight $(m_1,\ldots,m_r) \in \rr_{>0}^r$.
Let $\Ybar \to \Delta$ be the stable curve over $\Delta$ obtained by pulling back, in the category of analytic spaces, the stable curve $\Xbar \to U$ along the holomorphic arc $f \colon \Delta \to U$. Then in a neighborhood of a node $e$ of $\Xbar_0$, the surface $\Ybar$ is given by an equation $uv-t^{m(e)}$ and hence a local minimal resolution of singularities of $\Ybar$ at $e$ is obtained by replacing $e$ by a chain of $m(e)-1$ projective lines. Let $\Ybar' \to \Ybar$ be the minimal resolution of singularities of $\Ybar$. Then $\Ybar' \to \Delta$ is a semistable curve with smooth total space $\Ybar'$, and by construction of $\Ybar'$ the polarized metric graph obtained by taking the dual graph $(G',q')$ of the fiber of $\Ybar' \to \Delta$ over $0$ and giving each edge of $G'$ unit length is canonically isometric to the polarized metric graph $(G,q,m_1,\ldots,m_r)$.

With this in mind, Theorem \ref{main} reduces to the following statement.
\begin{thm} \label{reduced} Let $\Xbar$ be a smooth complex surface and let $\pi \colon \Xbar \to \Delta$ be a hyperelliptic semistable curve, smooth over $\Delta^*$. Let $(G,q)$ be the polarized dual graph of the special fiber $\Xbar_0$, and let $(\Gamma,q)$ be the associated polarized metric graph where each edge has unit length. Then the Zhang-Kawazumi invariant of the fibers $\Xbar_t$ satisfies the asymptotics
\[ \varphi(\Xbar_t) \sim -\varphi(\Gamma,q)\log|t| \]
as $t \to 0$.
\end{thm}
The remainder of this paper is devoted to a proof of Theorem \ref{reduced}. The proof consists of a number of steps. First we write the hyperelliptic Zhang-Kawazumi invariant in terms of the Faltings delta-invariant \cite{fa} and the Petersson norm of the modular discriminant \cite{lock}. This step was already accomplished in \cite{djsecond}. We then focus on the asymptotics of the Faltings delta-invariant and the Petersson norm of the modular discriminant. In \cite{djdelta} we have found the asymptotics of the Faltings delta-invariant in an arbitrary semistable family of curves over the unit disc. Next, a result of  I. Kausz \cite{ka} can be interpreted as giving the asymptotics of the Petersson norm of the modular discriminant. In our final step we relate the resulting asymptotics for $\varphi$ to Zhang's graph invariant via a result from \cite{ya}.

When $\pi \colon \Xbar \to \Delta$ is a semistable curve over the unit disc, we denote by $\omega$ the relative dualizing sheaf of $\pi$, and by $\lambda_1$ the determinant of the Hodge bundle $\lambda_1=\det R\pi_*\omega$ on $\Delta$.

We start by reviewing from \cite[Section ~3]{lock} the construction of a canonical discriminant modular form $\Delta_h$ on the moduli space $\ii_h$ of complex hyperelliptic curves of genus $h$. The form $\Delta_h$ generalizes the usual discriminant $\Delta$ of weight~$12$ occurring in the theory of moduli of elliptic curves. It is well known that for a semistable family of elliptic curves $\Xbar \to \Delta$, with $\Xbar$ smooth, the section
\[   \Lambda = (2\pi)^{12} \Delta(\tau)   (\d z)^{\otimes 12} \]
of the line bundle $\lambda_1^{\otimes 12}$ is trivializing over $\Delta^*$, and acquires a zero of multiplicity $\delta$ at the origin, where $\delta$ is the number of singularities of the special fiber $\Xbar_0$. These facts generalize to semistable families of hyperelliptic curves of genus $h$ over the unit disc as follows.

Let $n$ be the binomial coefficient $\binom{2h}{h+1}$. Let $\HH_h$ be Siegel's upper half-space consisting of symmetric complex $h \times h$-matrices with positive definite imaginary part. Then for $z$ in  $\cc^h$ (viewed as a column vector), matrices $\Omega \in \HH_h$ and $\eta',\eta''$ in $\frac{1}{2} \zz^h$ one has the classical theta function with characteristic
$\eta=\genfrac[]{0pt}{0}{\eta'}{\eta''}$ given by the Fourier series 
\[ \theta[\eta](z,\Omega) = \sum_{n \in \zz^h} \exp( \pi i  {}^t (n+\eta')
\Omega  (n+\eta') + 2\pi i  {}^t (n+\eta') (z+\eta'')) \, . \] 
For a given set $\tt$ of even theta characteristics in $\frac{1}{2}\zz^h$ we let $S(\tt)$ be the set of matrices $\Omega \in \HH_h$ such that the equivalence
\[ \vartheta[\eta](0,\Omega)\neq 0  \quad \Longleftrightarrow \quad \eta \in \tt \]
holds for $\Omega$. Following \cite[Theorem III.9.1]{mu} there exists a canonical set $\tt_0$ of even theta characteristics such that $S(\tt_0)$ is precisely the set of hyperelliptic period matrices. Here and in what follows, by a hyperelliptic period matrix we mean a normalized period matrix of a hyperelliptic Riemann surface $\Sigma$, formed on a canonical symplectic basis of $H_1(\Sigma,\zz)$. We refer to \cite[Section III.5]{mu} for the notion of a canonical symplectic basis, determined by an ordering of the hyperelliptic branch points of $\Sigma$.  

The discriminant modular form $\Delta_h$ is defined to be the function
\[ \Delta_h(\Omega) = 2^{-(4h+4)n} \prod_{\eta \in \tt_0}
\theta[\eta](0,\Omega)^8   \] on $\HH_h$. A verification shows that
$\Delta_h$ is a modular form on the congruence normal subgroup $\Gamma_h(2) \subset \mathrm{Sp}(2h,\zz)$ of weight $(8h+4)n/h$.
If $\Omega \in \HH_h$ is a hyperelliptic period matrix, we put
\begin{equation} \label{defnorm}
\|\Delta_h\|(\Omega) = (\det \mathrm{Im} \, \Omega)^{(4h+2)n/h}|\Delta_h(\Omega)| \, .
\end{equation}
Then for a given hyperelliptic curve $\Sigma$ of genus $h$ the value of $\|\Delta_h\|$ on a period matrix $\Omega$ of $\Sigma$ on a canonical basis of homology does not depend on the choice of $\Omega$.  We conclude that $\|\Delta_h\|$ induces a well-defined real-valued function on $\ii_h$. In \cite{djsecond} we found that the restriction of $\varphi$ to $\ii_h$ can be expressed in terms of $\|\Delta_h\|$ and the Faltings delta-invariant \cite{fa}.

\begin{thm} \label{ZKforhyp}
The Zhang-Kawazumi invariant of a hyperelliptic Riemann surface $\Sigma$ with hyperelliptic period matrix $\Omega$ satisfies the following formula:
\[ (2h-2) \varphi(\Sigma) =  -8(2h+1)h \log(2\pi) - 3(h/n) \log\|\Delta_h\|(\Sigma) - (2h+1)\delta_F(\Sigma) \, . \]
Here $\delta_F(\Sigma)$ is the Faltings delta-invariant of $\Sigma$.
\end{thm}  
\begin{proof} This is \cite[Corollary 1.8]{djsecond}.
\end{proof}
Let $\pi \colon \Xbar \to \Delta$ be a hyperelliptic semistable curve, with $\Xbar$ smooth over $\cc$, and with $\pi$ smooth over $\Delta^*$. 
Following the paper \cite{ka} by I. Kausz, the algebraic discriminant of a generic hyperelliptic equation for $\Xbar \to \Delta$ gives rise to a canonical global section $\Lambda$ of $\lambda_1^{8h+4}$ over $\Delta$ which is trivializing over $\Delta^*$. In holomorphic coordinates on the associated jacobian family we may write (cf. the proof of \cite[Theorem~8.2]{djmumford})
\begin{equation} \label{formulaLambda} \Lambda^{\otimes n} = (2\pi)^{(8h+4)hn} \Delta_h(\Omega)^{\otimes h} (\d z_1 \wedge \ldots \wedge \d z_h)^{\otimes (8h+4)n} \, .
\end{equation}
We say that a double point $x$ of the special fiber $\Xbar_0$ is of
type $0$ if the local normalization of $\Xbar_0$ at $x$ is connected. We
say that $x$ is
of type $i$, where $i=1,\ldots,[h/2]$, if the local normalization of
$\Xbar_0$ at $x$ is the disjoint union of a curve of  genus
$i$ and a curve of genus $h-i$. Let $\iota$ be the
involution on $\Xbar_0$ induced by the hyperelliptic involution on $\Xbar$.
Let $x$ be a double point of type $0$ on $\Xbar_0$. If $x$ is not fixed by $\iota$, the
local normalization of $\Xbar_0$ at $\{x,\iota(x)\}$ consists of two
connected components, of genus $j$ and $h-j-1$, say, where $0\leq j \leq
[(h-1)/2]$. In this case we say that
the pair $\{x,\iota(x)\}$ is of subtype $j$. Let $\xi'_0$ be the number of double points of type $0$ fixed by $\iota$, let $\xi_j$ for $j=0,\ldots,[(h-1)/2]$ be the number of
pairs $\{x,\iota(x)\}$ of double points of subtype $j$, and let $\delta_i$ for $i=1,\ldots,[h/2]$
be the number of double points of type $i$. Let the integer $d$ be given by
\begin{equation} \label{definitiond} d = h\xi'_0 + \sum_{j=0}^{[(h-1)/2]} 2(j+1)(h-j)\xi_j + \sum_{i=1}^{[h/2]} 4i(h-i)\delta_i \, .  
\end{equation}
\begin{thm} \label{cornalba}  Let $\Omega(t)$ be a family of hyperelliptic period matrices associated to $\Xbar \to \Delta$. Then the asymptotic formula
\begin{equation} \label{asymptDelta}
-(h/n)\log \|\Delta_h\|(\Xbar_t) \sim -d \log|t| - (4h+2)\log \det \Im \Omega(t)  
\end{equation}
holds as $t \to 0$.
\end{thm}
\begin{proof} 
By \cite[Theorem 3.1]{ka} we have, since $\Xbar$ is smooth, the identity $ \ord_0(\Lambda) = d$.
Combining with (\ref{formulaLambda}) we obtain the identity $\ord_0(\Delta_h) = dn/h$. Formula (\ref{asymptDelta}) then follows from (\ref{defnorm}).
\end{proof}
\begin{remark} The identity $\ord_0(\Lambda) = d$ is a refined version of an identity in $\mathrm{Pic}_\qq(\bar{\mathcal{I}}_h)$ due to M. Cornalba and J. Harris \cite[Proposition~4.7]{ch}. Here $\bar{\mathcal{I}}_h$ denotes the stack theoretic closure of $\mathcal{I}_h$ inside $\overline{\mm}_h$.
\end{remark}
Let $(G,q)$ be the dual graph of the special fiber $\Xbar_0$, and let $(\Gamma,q)$ be the associated polarized metric graph where each edge has unit length.
\begin{thm} \label{delta} 
Let $\delta$ be the total volume of $(\Gamma,q)$, and let $\vareps$ be Zhang's epsilon-invariant (\ref{defeps}) of $(\Gamma,q)$. Let $\Omega(t)$ be any family of normalized period matrices associated to $\Xbar \to \Delta$. Then the Faltings delta-invariant satisfies the asymptotics
\begin{equation} \label{asymptdelta} \delta_F(\Xbar_t) \sim -(\delta+\vareps) \log|t| - 6 \log \det \Im \Omega(t) 
\end{equation}
as $t \to 0$. 
\end{thm}
\begin{proof} This is a special case of \cite[Theorem~1.1]{djdelta}.
\end{proof}
\begin{proof}[Proof of Theorem \ref{reduced}]
Upon combining Theorems \ref{ZKforhyp}, \ref{cornalba} and \ref{delta} we obtain the asymptotic estimate
\[ (2h-2)\varphi(\Xbar_t) \sim -(3d-(2h+1)(\delta + \vareps)) \log|t|  \]
as $t \to 0$. Hence we are done once we show that the equality
\begin{equation} \label{reducedbis} (2h-2)\varphi  = 3d - (2h+1)(\delta+\vareps) 
\end{equation}
holds, where $\varphi=\varphi(\Gamma,q)$ is Zhang's graph invariant of $(\Gamma,q)$.  In \cite[Section 1.7]{ya} the invariant 
\[ \psi  = \vareps + \frac{2h-2}{2h+1}\varphi \]
is considered. With this notation formula (\ref{reducedbis}) is equivalent to the formula
\begin{equation} \label{tobeproven} (2h+1)\psi = 3d - (2h+1)\delta \, . \end{equation}
A combination of \cite[Section 1.9]{ya} and \cite[Theorem 3.5]{ya} yields that for the hyperelliptic polarized graph $(\Gamma,q)$ the equality
\begin{equation} \label{psiexplicit}  (2h+1)\psi   = (h-1)\delta_0 + \sum_{j=1}^{[(h-1)/2]} 6j(h-1-j)\xi_j 
  + \sum_{i=1}^{[h/2]} (12i(h-i)-(2h+1))\delta_i 
\end{equation}
holds. On the other hand, from (\ref{definitiond}) and the identities
\[ \delta = \delta_0 + \sum_{i=1}^{[h/2]} \delta_i \, , \quad \delta_0 = \xi'_0 + 2 \sum_{j=0}^{[(h-1)/2]} \xi_j \]
we obtain
\begin{equation} \label{combiexplicit}  3d - (2h+1)\delta  = (h-1)\delta_0 + \sum_{j=0}^{[(h-1)/2]} (6(j+1)(h-j)-6h) \xi_j 
  + \sum_{i=1}^{[h/2]} (12i(h-i)-(2h+1))\delta_i \, . 
\end{equation}
One readily checks that the right hand sides of (\ref{psiexplicit}) and (\ref{combiexplicit}) are equal, and (\ref{tobeproven}) follows. This finishes the proof of Theorem \ref{reduced}.
\end{proof}

\subsection*{Acknowledgments} I would like to thank Boris Pioline for helpful discussions.

\vspace{0.5cm}

\noindent Address of the author:\\ \\
Mathematical Institute  \\
Leiden University \\
PO Box 9512  \\
2300 RA Leiden  \\
The Netherlands  \\ \\
Email: \verb+rdejong@math.leidenuniv.nl+


\begin{thebibliography}{99}

\bibitem{abbf} O. Amini, S. Bloch, J. Burgos Gil, J. Fres\'an, \emph{Feynman amplitudes and limits of heights}. To appear.

\bibitem{ar} S. Y. Arakelov, \emph{An intersection theory for divisors on an
arithmetic surface}. Izv. Akad. USSR 86 (1974), 1164--1180.

\bibitem{br} M. Baker, R. Rumely, \emph{Harmonic analysis on metric graphs}. Canad. J. Math. 59, no. 2 (2007), 225--275.

\bibitem{bloch} S. Bloch, \emph{Feynman amplitudes in mathematics and physics}. Preprint \verb+arxiv:1509.00361+.

\bibitem{ca} E. Cattani,  \emph{Mixed Hodge structures, compactifications and monodromy weight filtration}. In: Topics in transcendental algebraic geometry (Princeton, N.J., 1981/1982), 75--100, Ann. of Math. Stud., 106, Princeton Univ. Press, Princeton, NJ, 1984. 

\bibitem{cr} T. Chinburg, R. Rumely, \emph{The capacity pairing}. J. reine angew. Math. 434 (1993), 1--44.

\bibitem{ci3} Z. Cinkir, \emph{Admissible invariants of genus 3 curves}. Preprint \verb+arxiv:1405.7413+.

\bibitem{cicomp} Z. Cinkir, \emph{Computation of polarized metrized graph invariants by using discrete Laplacian matrix}. Math. Comp. 84 (2015), no. 296, 2953--2967.

\bibitem{ci} Z. Cinkir, \emph{Zhang's conjecture and the effective Bogomolov
conjecture over function fields}. Invent. Math. 183 (2011), 517--562.

\bibitem{ch} M. Cornalba, J. Harris, \emph{Divisor classes associated to
families of stable varieties, with applications to the moduli space of curves}.
Ann. Scient. Ec. Num. Sup. 21 (1988), no. 4, 455--475.

\bibitem{fab} X. Faber, \emph{The geometric Bogomolov conjecture for curves of small genus}. Experiment. Math. 18 (2009), no. 3, 347--367.

\bibitem{fa} G. Faltings, \emph{Calculus on arithmetic surfaces}.  
Ann. of Math. 119  (1984),  no. 2, 387--424. 

\bibitem{dhgr} E. D'Hoker and M. B. Green, \emph{Zhang-Kawazumi invariants and superstring amplitudes}. J. Number Theory 144 (2014), 111--150.

\bibitem{dhgrpr} E. D'Hoker, M. B. Green, B. Pioline, R. Russo, \emph{Matching the $D^6R^4$ interaction at two-loops}. J. High Energy Phys. 1501 (2015).

\bibitem{djdelta} R. de Jong, \emph{Faltings delta-invariant and semistable degeneration}. Preprint \verb+arxiv:1511.06567+.

\bibitem{djtorus} R. de Jong, \emph{Torus bundles and 2-forms on the universal family of Riemann surfaces}. Preprint \verb+arxiv:1309.0951+. To appear in Handbook of Teichm\"uller theory, Vol. V, Eur. Math. Soc., Z\"urich. 

\bibitem{djasympt} R. de Jong, \emph{Asymptotic behavior of the Kawazumi-Zhang invariant for degenerating Riemann surfaces}. Asian J. Math. 18
(2014), 507--524.

\bibitem{djnormal} R. de Jong, \emph{Normal functions and the height of Gross-Schoen cycles}. Nagoya Math. Jnl. 213 (2014), 53--77. 

\bibitem{djsecond}  R. de Jong, \emph{Second variation of Zhang's $\lambda$-invariant on the moduli space of curves}. Amer. Jnl. Math. 135 (2013), 275--290.

\bibitem{djgenus2} R. de Jong, \emph{Admissible constants for genus 2 curves}. Bull. London Math. Soc. 42 (2010), 405--411.

\bibitem{djmumford} R. de Jong, \emph{Explicit Mumford isomorphism for hyperelliptic curves}. Jnl. pure appl. Algebra 208 (2007), 1--14.

\bibitem{ka} I. Kausz, \emph{A discriminant and an upper bound for
$\omega^2$ for hyperelliptic arithmetic surfaces}.
Compositio Math.  115  (1999),  no. 1, 37--69.

\bibitem{kawhandbook} N. Kawazumi, \emph{Canonical 2-forms on the moduli space of Riemann surfaces}. Handbook of Teichm\"uller theory. Vol. II, 217--237, IRMA Lect. Math. Theor. Phys., 13, Eur. Math. Soc., Z\"urich, 2009. 

\bibitem{kaw} N. Kawazumi, \emph{Johnson's homomorphisms and the Arakelov-Green
function}. Preprint \verb+arxiv:0801.4218+.

\bibitem{lock} P. Lockhart, \emph{On the discriminant of a hyperelliptic
curve}. Trans. Amer. Math. Soc. 342 (1994), no. 2, 729--752.

\bibitem{mo} S. Morita, \emph{The extension of Johnson's homomorphism from the Torelli group to the mapping class group}. Invent. Math. 111 (1993), 197--224.

\bibitem{mu} D. Mumford, \emph{Tata lectures on Theta}. Birkh\"auser, Boston, Mass., 1983.

\bibitem{pio} B. Pioline, \emph{A theta lift representation for the Kawazumi-Zhang and Faltings invariants of genus-two Riemann surfaces}. Preprint \verb+arxiv:1504.04182+.

\bibitem{to} P. Tourkine, \emph{Tropical amplitudes}. Preprint \verb+arxiv:1309.3551+.

\bibitem{yacorn} K. Yamaki, \emph{Cornalba-Harris equality for semistable
hyperelliptic curves in positive characteristic}.
 Asian J. Math.  8  (2004),  no. 3, 409--426.

\bibitem{ya} K. Yamaki, \emph{Graph invariants and the positivity of the height of the
Gross-Schoen cycle for some curves}. manuscripta math. 131 (2010), 149--177.

\bibitem{zh} S. Zhang, \emph{Admissible pairing on a curve}. Invent. Math. 112 (1993), 171--193.

\bibitem{zhgs} S. Zhang, \emph{Gross-Schoen cycles and dualizing sheaves}.
Invent. Math. 179 (2010), 1--73.


\end{thebibliography}
\end{document}